\documentclass[a4paper,12pt]{article}

\usepackage{authblk}
\usepackage{amsmath,amsthm,amsfonts,amssymb}
\usepackage{bbm}
\usepackage[left=2.5cm, right=2.5cm, top=2cm, bottom=3.5cm]{geometry}
\usepackage[noadjust]{cite}
\usepackage[pdfstartview=FitH,
			CJKbookmarks=true,
			bookmarksnumbered=true,
			colorlinks, 
			linkcolor=blue,
			anchorcolor=blue,
			citecolor=blue,
			urlcolor=blue
			]{hyperref}

\newtheorem{thm}{Theorem}

\newtheorem{lem}[thm]{Lemma}

\def\a{\alpha}
\def\b{\beta}
\def\l{\lambda}
\def\g{\gamma}
\def\z{\zeta}
\newcommand{\df}{{\rm d}}
\newcommand{\nb}{\nabla}
\newcommand{\lp}{\ell_P}
\newcommand{\rnl}{\widehat N}
\newcommand{\jap}{J_\a^p}

\newcommand{\R}{\mathbb{R}}
\newcommand{\rn}{\R^n}
\newcommand{\jz}{\mathbb{M}^n}
\newcommand{\sph}{S^{n-1}}
\newcommand{\hsd}{\mathcal H^{n-1}}
\newcommand{\sln}{{\rm SL}(n)}
\newcommand{\gln}{{\rm GL}(n)}

\newcommand{\kn}{\mathcal K^n}
\newcommand{\kno}{\mathcal K^n_o}
\newcommand{\polyo}{\mathcal P^n_o}

\newcommand{\sobo}{W^{1,p}(\rn)}
\newcommand{\sobop}{W^{1,p}(\rn)}
\newcommand{\soboe}{W^{1,2}(\rn)}
\newcommand{\lip}{L^{1,p}(\rn)}
\newcommand{\pip}{P^{1,p}(\rn)}
\newcommand{\lprn}{L^p(\rn)}
\newcommand{\bp}{\mathcal B_p}

\newcommand{\zlt}{\mathbbm{t}}
\newcommand{\symt}{{\rm Sym}^p(\rn)}
\newcommand{\mop}{M^{0,p}}
\newcommand{\mpo}{M^{p,0}}
\newcommand{\ten}[2]{{#1}^{\odot{#2}}}
\newcommand{\tenp}[1]{{#1}^{\odot p}}

\newcommand{\set}[1]{\left\{{#1}\right\}}						%set
\newcommand{\xset}[1]{\{{#1}\}}										%small set
\newcommand{\abs}[1]{\left\vert{#1}\right\vert}				%absolute value
\newcommand{\norm}[1]{\left\Vert{#1}\right\Vert}		%norm
\newcommand{\normk}[2]{\left\Vert{#2}\right\Vert_{#1}}
\newcommand{\inp}[2]{\left\langle#1,#2\right\rangle}	%inner product

\makeatletter
\newcommand{\subjclass}[2][1991]{%
	\let\@oldtitle\@title%
	\gdef\@title{\@oldtitle\footnotetext{#1 \emph{Mathematics subject classification.} #2.}}%
}

\newcommand{\keywords}[1]{%
	\let\@@oldtitle\@title%
	\gdef\@title{\@@oldtitle\footnotetext{\emph{Key words and phrases.} #1.}}%
}
\makeatother

\title{{\bf Tensor valuations on Sobolev spaces}} 
\author{Tian Gao and Dan Ma\thanks{Corresponding author: \href{mailto:madan@shnu.edu.cn}{madan@shnu.edu.cn}}\\
	Department of Mathematics, Shanghai Normal University, Shanghai 200234, China}

\date{}

\begin{document}
	
\maketitle
	
\begin{abstract}
A complete classification is established for continuous, \(\sln\) contravariant, and translation invariant tensor valuations defined on the Sobolev space \(\sobo\).
When these valuations are further assumed to be homogeneous, the classification reveals that they are precisely the Fisher information tensors, 
which constitute a higher-order generalization of the Fisher information matrix.
\end{abstract}

%%%%%%%%%%%%%%%%%%%%%%%%%%%%%%%%%%%%%%%%%%%%%%%%%%%%%%%%%%%%%%%%%%%%%%%%%%%%%%%%
%%%%%%%%%%%%%%%%%%%%%%%%%%%%%%%%%%%%%  Sections %%%%%%%%%%%%%%%%%%%%%%%%%%%%%%%%%%%%%
%%%%%%%%%%%%%%%%%%%%%%%%%%%%%%%%%%%%%%%%%%%%%%%%%%%%%%%%%%%%%%%%%%%%%%%%%%%%%%%%
	
\section{Introduction}

Let \(\mathcal S\) be a lattice of real-valued functions and let \(\langle\mathcal A,+\rangle\) be an Abelian semigroup. A map \(Z:\mathcal S\to\mathcal A\) is a valuation when
\begin{equation}\label{val}
	Z(f\vee g)+Z(f\wedge g)=Z(f)+Z(g)
\end{equation}
for every \(f,g\in\mathcal S\), where \(f\vee g\) and \(f\wedge g\) denote the pointwise maximum and minimum, respectively. For valuations on sets, the same identity is written with union and intersection in place of pointwise maximum and minimum.
This concept originated from Dehn's solution of Hilbert's third problem in 1901.
Later, Hadwiger established that every continuous, rigid motion invariant real-valued valuation on convex bodies decomposes into a linear combination of intrinsic volumes. This foundational result has since become a cornerstone of integral geometry (see \cite{1997KR}). 
For new results on valuations on convex sets, see \cite{2012Haberl,2010LR,2014HP,2017LM,2003Monika,2017LS}.

Recently, the study of valuations has gradually extended from spaces of sets to function spaces. 
Classification theory has already been developed on Sobolev and \(L^p\) lattices \cite{2011Monika,2012Monika,2016Ma,2010Tsang,2012Tsang,2013Monika,2017LM,2021WHL,2024WTL,2025LZ}. 
Related theories for convex, Lipschitz, and definable functions are treated in
\cite{2013BGW,2020CPTV,2021CPTV,2022CLM,2023CLM,2024CLM,2025CLM}.
Here we study the tensor-valued valuations (briefly tensor valuations) on the Sobolev space \(W^{1,p}(\rn)\).

Throughout the paper, let \(n\geq3\) and let \(p\) be an integer satisfying \(2\leq p<n\). The same integer \(p\) specifies both the Sobolev exponent and the tensor order. We write \(\symt\) for the space of symmetric tensors of order \(p\) on \(\rn\). 
For \(x\in\rn\), set
\[x^{\odot p}=\underbrace{x\odot\cdots\odot x}_{p\text{ factors}},\]
with \(\odot\) representing the symmetric tensor product. The gradient tensor to be characterized is
\begin{equation}\label{eqn:gradient-tensor}
	\mathcal I_p(f)=\int_{\rn}(\nb f(x))^{\odot p}\,\df x,
	\qquad f\in W^{1,p}(\rn).
\end{equation}
The integral is finite since \(\nb f\in L^p(\rn)\). The weak gradient identity above shows that \(\mathcal I_p\) satisfies \eqref{val}, and its continuity on \(W^{1,p}(\rn)\) follows from H\"older's inequality.
Related tensor valuations on convex bodies, including those covariant under isometries or the special linear group, are studied in
\cite{1997McMullen,1999AleskerAnn,1999Alesker,2010LR,2014HP,2016HP,2017HP,2014HS,2017JK,2017LS,2021BDS,2019ABDK}.

The order-two case of \eqref{eqn:gradient-tensor} is related to the classical Fisher information matrix. For a weakly differentiable function \(g:\rn\to[0,\infty)\) for which the integral below is finite, let \(J(g)\) be the matrix with entries
\[J_{ij}(g)=\int_{\rn}\frac{\partial\log g(x)}{\partial x_i}\frac{\partial\log g(x)}{\partial x_j}g(x)\,\df x.\]
For \(f\in W^{1,2}(\rn)\), it gives
\begin{equation*}\label{eqn:fisher-matrix}
	J_{ij}(f^2)=4\int_{\rn}
	\frac{\partial f(x)}{\partial x_i}
	\frac{\partial f(x)}{\partial x_j}\,\df x.
\end{equation*}
Lutwak, Yang, and Zhang initiated the study of this order-two Fisher information tensor in affine convex geometry. In \cite{2000LYZ}, they introduced the ellipsoid now known as the LYZ ellipsoid. 
They later established a correspondence between this ellipsoid and the one determined by the Fisher information matrix, and obtained a convex-geometric counterpart of the Cram\'{e}r--Rao inequality \cite{2002LYZ}.

Using this connection, Ludwig completely classified the continuous affinely contravariant matrix-valued valuations on \(W^{1,2}(\rn)\) in \cite{2011Monika}. 
To state her result, let \(\jz\) be the space of symmetric real \(n\times n\) matrices. 
A map \(Z:\soboe\to\jz\) is called \(\gln\) contravariant if for some \(r\in\R\),
\[Z(f\circ\phi^{-1})=\abs{\det\phi}^r\phi^{-t}Z(f)\phi^{-1}\]
for every \(f\in\soboe\) and \(\phi\in\gln\), where \(\det\phi\) is the determinant of \(\phi\)
and \(\phi^{-t}\) is the inverse of the transpose of \(\phi\). It is called translation invariant if
\(Z(f\circ\tau^{-1})=Z(f)\) for every \(f\in\soboe\) and translation \(\tau\),
and homogeneous if for some \(q\in\R\), \(Z(sf)=\abs{s}^qZ(f)\) for every \(f\in\soboe\) and \(s\in\R\). The map \(Z\) is called affinely contravariant if it has all three properties.
\begin{thm}[\cite{2011Monika}]\label{thm:Ludwig}
	Let \(n>2\). A map \(Z:W^{1,2}(\rn)\to\jz\) is a continuous, affinely
	contravariant valuation if and only if there exists \(c\in\R\) such that
	\[
	Z(f)=cJ(f^2)
	\]
	for every \(f\in W^{1,2}(\rn)\).
\end{thm}

Motivated by this order-two result, we consider the \(p\)-th symmetric moment of \(\nb \log g\). For a sufficiently regular positive function \(g\), set
\[
J^p(g)=\int_{\rn}\bigl(\nb\log g(x)\bigr)^{\odot p}g(x)\,\df x.
\]
We refer to this Fisher-type quantity as the Fisher information tensor of order
\(p\). If \(f\geq0\) and \(g=f^p\), then the chain rule gives
\begin{equation}\label{eqn:jap}
	J^p(f^p)=p^p\mathcal I_p(f)
	=p^p\int_{\rn}(\nb f(x))^{\odot p}\,\df x.
\end{equation}
For arbitrary \(f\in W^{1,p}(\rn)\), we use the right-hand side of
\eqref{eqn:jap} to define \(J^p(f^p)\). For \(p=2\), it agrees with
the classical Fisher information matrix above.
We now consider maps \(Z:\sobop\to\symt\). Such a map is continuous if
\(Z(f_k)\to Z(f)\) whenever \(f_k\to f\) in \(\sobo\) as \(i\to\infty\). 
It is called \(\sln\) contravariant if \(Z(f\circ\phi^{-1})=\phi^{-t}\cdot Z(f)\) for every \(f\in\sobop\) and \(\phi\in\sln\).
It is called translation invariant if \(Z(f\circ\tau^{-1})=Z(f)\) for every \(f\in\sobop\) and translation \(\tau\),
and homogeneous of degree \(q\in\R\) if \(Z(sf)=s^qZ(f)\) for every \(f\in\sobop\) and \(s\in\R\setminus\{0\}\).
Our main results characterize the maps with these properties and give the order-\(p\) analogue of Theorem \ref{thm:Ludwig}.
\begin{thm}\label{thm:m1}
	A map \(Z:\sobop\to\symt\) is a continuous, \(\sln\) contravariant, translation invariant and \(p\)-homogeneous valuation
	if and only if there is a constant \(c\in\R\) such that
	\[Z(f)=cJ^p(f^p)\]
	for every \(f\in\sobop\).
\end{thm}

Let \(\a:\R\to\R\) be a bounded function. Due to the chain rule, it suffices to consider the following generalization in the non-homogeneous case
\begin{equation*}
	\jap(f^p)=p^p\int_{\rn}\left(\nb f(x)\right)^{\odot p}\a(f(x))\df x.
\end{equation*}
Let \(Z:\sobo\to\symt\), and let \(\Phi(s)=Z(s\lp)\) for every \(s\in\R\) and \(\lp\in\pip\) (see Section \ref{sec:two} for definitions).
We write \(\bp\) for the collection of \(Z\) such that \(\Phi\in C^n(\R)\) and \(s^{k-p}\Phi^{(k)}(s)\) is bounded for \(k=0,1,2, \dots, n\).
We obtain a more general classification in this case.
\begin{thm}\label{thm:m2}
	A map \(Z:\sobop\to\symt\) is a continuous, \(\sln\) contravariant, translation invariant valuation
	with \(Z\in\bp\) and \(Z(0)=0\)
	if and only if there is a continuous and bounded function \(\a:\R\to\R\) such that	\[Z(f)=\jap(f^p)\]
	for every \(f\in\sobop\).
\end{thm}

%%%%%%%%%%%%%%%%%%%%%%%%%%%%%%%%%%%%%%%%%%%%%%%%%%%%%%%%%%%%%%%%%%%%%%%%%%%%%%%%
%%%%%%%%%%%%%%%%%%%%%%%%%%%%%%%%%%%%%  Sections %%%%%%%%%%%%%%%%%%%%%%%%%%%%%%%%%%%%%
%%%%%%%%%%%%%%%%%%%%%%%%%%%%%%%%%%%%%%%%%%%%%%%%%%%%%%%%%%%%%%%%%%%%%%%%%%%%%%%%

\section{Preliminaries}\label{sec:two}

We work in \(n\)-dimensional Euclidean space \(\rn\) with standard scalar product \(\inp{\cdot}{\cdot}\) and induced norm \(\abs{\cdot}\).
We write \(\sph=\set{x\in\rn:\abs{x}=1}\) for the unit sphere of \(\rn\).
The \(k\)-dimensional Hausdorff measure in Euclidean spaces is denoted by \(\mathcal H^k\).

A symmetric tensor on \(\rn\) of order \(p\) is defined as a \(p\)-multilinear function from \((\rn)^p\) to \(\R\).
The symmetric tensor product of tensors \(\zlt_i\in{\rm Sym}^{p_i}(\rn)\) for \(i=1,\dots,k\) is
\[(\zlt_1\odot\dots\odot\zlt_k)(v_1,\dots,v_p)=\frac1{p!}\sum_\sigma(\zlt_1\otimes\dots\otimes\zlt_k)(v_{\sigma(1)},\dots,v_{\sigma(p)})\]
for \(v_1,\dots,v_p\in\rn\), where \(p=p_1+\dots+p_k\), the ordinary tensor product is denoted by \(\otimes\),
and we sum over all of the permutations of \(\set{1,\dots,p}\).
For \(\zlt\in\symt\) and \(\phi\in\gln\), we denote by \(\phi\cdot\zlt\) the natural action of \(\phi\) to \(\zlt\). That is,
\[(\phi\cdot\zlt)(v_1,\cdots,v_p)=\zlt(\phi^t v_1,\cdots,\phi^t v_p),\]
where \(\phi^t\) is the transpose of \(\phi\).
Let \(\set{e_1,\cdots,e_n}\) be the standard basis of \(\rn\). Then the tensors \(e_{i_1}\odot\cdots\odot e_{i_p}\) with \(1\leq i_1\leq\cdots\leq i_p\leq n\) form a basis of \(\symt\).
This allows us to write \(\zlt\in\symt\) as
\[\zlt=\sum_{1\leq i_1\leq\cdots\leq i_p\leq n}t_{i_1\ldots i_p}e_{i_1}\odot\cdots\odot e_{i_p},\]
where \(t_{i_1\cdots i_p}=\binom{p}{m_1 \ldots m_n}\zlt(e_{i_1},\cdots,e_{i_p})\),
\(m_k\) counts how often the number \(k\) appears among the indices \(i_1,\ldots,i_p\) for \(k = 1,\ldots,n\).
We consider the Euclidean topology with the corresponding dimension, and thus the norm is given by
\[\norm{\zlt}=\left(\sum_{1\leq i_1\leq\cdots\leq i_p\leq n}t_{i_1\ldots i_p}^2\right)^{1/2}.\]
In particular, taking \(\zlt=x^{\odot p}\) for \(x\in\rn\), we have \(\zlt(e_{i_1},\cdots,e_{i_p})=\inp{x}{e_{i_1}}\cdots\inp{x}{e_{i_p}}=x_{i_1}\cdots x_{i_p}\).
Since  \(x_{i_k}^2\leq\abs{x}^2\) for every \(k=1,\cdots,p\), we have \(\norm{\zlt}\leq\l_{n,p}\abs{x}^p\), where
\(\l_{n,p}=\left(\sum_{1\leq i_1\leq\cdots\leq i_p\leq n}\binom{p}{m_1 \ldots m_n}^2\right)^{\frac12}\).

Next, we collect basics regarding convex bodies, which are referred to \cite{2014Sch}.
Let \(\kn\) denote the space of convex bodies, i.e., compact convex subsets in \(\rn\),
and \(\kno\) be the subspace of \(\kn\) whose elements contain the origin in their interiors.
Let \(\polyo\) denote the space of convex polytopes containing the origin in their interiors.

For \(K\in\kn\), define its support function \(h(K,\cdot):\rn\to\R\) by \[h(K,x)=\max\{\inp{x}{y}:y\in K\},\qquad x\in\rn.\]
For \(K\in\kno\), its radial function \(\rho(K,\cdot):\rn\setminus\{o\}\to\R\) is defined by
\[\rho(K,x)=\max\{t>0 : tx\in K\},\qquad x\in\R^n\setminus\{o\}.\]
And the polar body \(K^{\ast}\) of \(K\) is defined by
\[K^{\ast}=\{x\in\rn:\inp{x}{y}\le1 ~\text{for all}~ y\in K\}.\]
Let \(\partial K\) be the boundary of \(K\) and \(N_K(x)\) be the outer unit normal vector at \(x\in\partial K\). Then \[\nb h(K^\ast,x)=\frac{N_K(x)}{h(K,N_K(x))}.\]

For \(P\in\polyo\), the map \(\mpo:\polyo\to\symt\) is defined by \[\mpo(P)=\int_Px^{\odot p}\df x,\]
and the map \(\mop:\polyo\to\symt\) is defined by \[\mop(P)=\int_{\sph}u^{\odot p}\df S_p(P,u),\]
where \(S_p(P,\cdot)\) is the \(L_p\) surface area measure of \(P\). When \(p=1\), \(S_1(P,\cdot)\) corresponds to the classical surface area measure of \(P\).
A map \(\mu:\polyo\to\symt\) is called \(\sln\) contravariant if \(\mu(\phi P)=\phi^{-t}\cdot\mu(P)\) for every \(\phi\in\sln\) and \(P\in\polyo\). It is called Borel measurable if the preimage of every open set is a Borel set.
We make essential use of the following classification of valuations on polytopes established in \cite{2017HP}.
\begin{thm}[\cite{2017HP}]\label{thm:body}
	Let \(p\geq 2\) and \(n\geq3\). A map \(\mu:\polyo\to\symt\) is a Borel measurable and \(\sln\) contravariant valuation
	if and only if there are constants \(c_1,c_2\in\R\) such that
	\[\mu(P)=c_1M^{p,0}(P^\ast)+c_2\mop(P)\]
	for every \(P\in\polyo\).
\end{thm}

Let \(1\leq p<\infty\). For a measurable function \(f:\rn\to\R\), define its \(L^p\)-norm as
\[\normk{p}{f}=\left(\int_{\rn}\abs{f(x)}^p\df x\right)^{1/p}.\] 
The Banach space \(L^p(\rn)\) is the set of measurable functions with finite \(L^p\)-norms.
Let \(\sobo\) denote the Sobolev space of all functions \(f\in\lprn\) whose weak gradients also belong to \(\lprn\).
Here, a measurable function \(\nb f:\rn\to\rn\) is called the weak gradient of \(f\in\lprn\) if 
\[\int_{\rn}\inp{\nu(x)}{\nb f(x)}\df x=-\int_{\rn}f(x)\nb\cdot\nu(x)\df x\]
for every compactly supported smooth vector field \(\nu:\rn\to\rn\), where \(\nb\cdot\nu=\frac{\partial\nu_1}{\partial x_1}+\cdots+\frac{\partial\nu_n}{\partial x_n}\).
For every \(f\in\sobo\), denote the Sobolev norm by
\[\normk{\sobo}{f}=\left(\normk{p}{f}^p+\normk{p}{\nb f}^p\right)^{1/p},\] where \(\normk{p}{\nb f}\) denotes the \(L^p\)-norm of \(\abs{\nb f}\).

For \(f,g\in\sobo\), the set \(\set{x\in\rn: f(x)=g(x)\text{ and }\nb f(x)\neq\nb g(x)}\) has measure zero (see \cite{2001LL}).
Thus, we have \(f\vee g, f\wedge g\in\sobo\) and for almost every \(x\in\rn\),
\begin{equation}\label{fvg}
	\nb(f\vee g)(x)=
	\begin{cases}
		\nb f(x), &\text{if~}f(x)>g(x),\\
		\nb g(x), &\text{if~}f(x)<g(x),\\
		\nb f(x)=\nb g(x), &\text{if~}f(x)=g(x), 
	\end{cases}
\end{equation}
and
\begin{equation}\label{f^g}
	\nb(f\wedge g)(x)=
	\begin{cases}
		\nb f(x), &\text{if~}f(x)<g(x),\\
		\nb g(x), &\text{if~}f(x)>g(x),\\
		\nb f(x)=\nb g(x), &\text{if~}f(x)=g(x).
	\end{cases}
\end{equation}
Hence \((\sobo, \vee, \wedge)\) is a lattice.

A function \(\ell:\rn\to\R\) is called piecewise affine, if it is continuous and there exist finitely many \(n\)-dimensional simplices \(\Delta_1,\cdots,\Delta_k\subset\rn\)
with pairwise disjoint interiors such that \(\ell|_{\Delta_i}\) is affine for \(i=1, \cdots, k\) and \(\ell\) vanishes outside \(\Delta_1\cup\cdots\cup\Delta_k\).
We denote by \(\lip\) the linear space of all such functions. It is immediate from the definition that \(\lip\) is a subspace of \(\sobo\).
These simplices are called a triangulation of the support of \(\ell\).
The function \(\ell\) is uniquely determined by its values on the vertex set \(V\) of this triangulation.
Moreover, \(\lip\) is dense in \(\sobo\) (see \cite{2009Leoni}).

There is an elementary yet important class of functions as examples.
Let \(P\in\polyo\). We define
\[\lp(x)=\left(1-h(P^\ast,x)\right)\vee0\]
for \(x\in P\) and call them cone functions. The subset of \(\lip\) consisting of all functions \(\lp\) for \(P\in\polyo\) is denoted by \(\pip\).
For every \(\phi\in\gln\), we have \(\ell_{\phi P}=\lp\circ\phi^{-1}\).
It also follows that
\[\nb\lp(x)=-\nb h(P^\ast,x)=-\frac{\rnl_P(x)}{h(P,\rnl_P(x))},\]
where \(\rnl_P(x)=N_P(\rho(P,x)x)\).
We note that multiples and translates of \(\lp\in\pip\) correspond to linear elements within the theory of finite elements (see \cite{2008BS}).

The coarea formula states that if \(f:\rn\to\R\) is a Lipschitz continuous function and \(g:\rn\to\R\) is integrable, then
\[\int_{\rn}g(x)\abs{\nb f(x)}\df x=\int_{\R}\left(\int_{\set{f=t}}g(x)\df \hsd(x)\right)\df t,\] 
where \(\set{f=t}=\set{x\in\rn:f(x)=t}\) is the level set.

Making use of the coarea formula, we have the following calculation.
\begin{lem}\label{thm:cal}
	Let \(s\in\R\) and \(P\in\polyo\). Then for a continuous and bounded function \(\a:\R\to\R\), we have
	\[\jap((s\lp)^p)=(-p)^ps^{p-n}\mop(P)\int_0^s\a(t)(s-t)^{n-1}\df t.\]
\end{lem}
\begin{proof}
	First, for every \(x\in P\), since the gradient of the support function is 0-homogeneous, we obtain
	\begin{align*}
		\abs{\nb h(P^\ast,x)} &= \abs{\nb h(P^\ast,\rho(P,x)x)}
		= \abs{\frac{N_P(\rho(P,x)x)}{h(P,N_P(\rho(P,x)x))}}\\
		&= h(P,N_P(\rho(P,x)x))^{-1}
		= h(P,\rnl_P(x))^{-1}.
	\end{align*}
	Combined with the construction of the cone function, we obtain
	\begin{align*}
		\jap((s\lp)^p) &= (ps)^p\int_{\rn}\tenp{\nb\lp(x)}\a(s\lp(x))\df x\\
		&= (-ps)^p\int_P\tenp{\rnl_P(x)}h(P,\rnl_P(x))^{-p}\a(s(1-h(P^\ast,x)))\df x\\
		&= (-ps)^p\int_P\tenp{\rnl_P(x)}h(P,\rnl_P(x))^{1-p}\abs{\nb h(P^\ast,x)}\a(s(1-h(P^\ast,x)))\df x.
	\end{align*}
	We use the coarea formula to proceed
	\begin{equation}\label{eqn:cf}
		(-ps)^p\int_0^1\int_{t\partial P}\tenp{\rnl_P(x)}h(P,\rnl_P(x))^{1-p}\a(s(1-t))\df\hsd(x)\df t.
	\end{equation}
	Next, for every \(t\in(0,1)\) and \(y\in\partial P\), we notice
	\begin{align*}
		N_P(y)=N_P(\rho(P,y)y)=N_P(\rho(P,ty)ty)=\rnl_P(ty),
	\end{align*}
	and obtain that \eqref{eqn:cf} equals
	\begin{align*}
		& (-ps)^p\int_0^1\int_{\partial P}\tenp{N_P(y)}h(P,N_P(y))^{1-p}\a(s(1-t))t^{n-1}\df\hsd(y)\df t\\
		=& (-ps)^p\left(\int_0^1t^{n-1}\a(s-st)\df t\right)\left(\int_{\sph}\tenp{u}h(P,u)^{1-p}\df S(P,u)\right)\\
		=& (-p)^ps^{p-n}\mop(P)\int_0^s\a(t)(s-t)^{n-1}\df t.
	\end{align*}
\end{proof}

Indeed, we have the following reduction to obtain the classification of continuous valuations on \(\sobop\). 
The proof is omitted since it is similar to \cite[Lemma 3.1]{2016Ma}.
\begin{lem}[\cite{2016Ma}]\label{thm:scm}
	Let \(Z_1,Z_2:\lip\to\symt\) be continuous and translation invariant valuations satisfying \(Z_1(0)=Z_2(0)=0\).
	If \(Z_1(sf)=Z_2(sf)\) for every \(s\in\R\) and \(f\in\pip\), then
	\[Z_1(f)=Z_2(f)\]
	for every \(f\in\lip\).
\end{lem}

%%%%%%%%%%%%%%%%%%%%%%%%%%%%%%%%%%%%%%%%%%%%%%%%%%%%%%%%%%%%%%%%%%%%%%%%%%%%%%%%
%%%%%%%%%%%%%%%%%%%%%%%%%%%%%%%%%%%%%  Sections %%%%%%%%%%%%%%%%%%%%%%%%%%%%%%%%%%%%%
%%%%%%%%%%%%%%%%%%%%%%%%%%%%%%%%%%%%%%%%%%%%%%%%%%%%%%%%%%%%%%%%%%%%%%%%%%%%%%%%

\section{Proofs of the main results}\label{sec:san}

First, we show the ``if'' part of Theorem \ref{thm:m2} in the following lemma.
\begin{lem}\label{thm:if}
	Let \(\a:\R\to\R\) be a continuous and bounded function. Then, for every \(f\in\sobop\), the map \(Z:\sobop\to\symt\), defined by
	\[Z(f)=\jap(f^p),\]
	is a continuous, \(\sln\) contravariant and translation invariant valuation with \(Z\in\bp\) and \(Z(0)=0\).
\end{lem}
\begin{proof}
	First, since \(\a\) is bounded, there exists \(M>0\) such that
	\[\norm{Z(f)}\leq p^p\int_{\rn}\norm{\tenp{\nb f}\a(f)}\leq p^p\l_{n,p}M\int_{\rn}\abs{\nb f}^p<\infty\]
	for every \(f\in\sobop\). Hence, \(Z\) is well defined.
	
	Next, let \(f\in\sobop\) and \(\set{f_i}_{i=1}^\infty\) be a sequence in \(\sobop\) that converges to \(f\). We have
	\begin{align*}
		\norm{Z(f_i)-Z(f)} &\leq p^p\norm{\int_{\rn}\left(\tenp{\nb f_i}-\tenp{\nb f}\right)\a(f_i)}+p^p\norm{\int_{\rn}\tenp{\nb f}\left(\a(f_i)-\a(f)\right)}\\
		&:= p^pI_1(f_i)+p^pI_2(f_i).
	\end{align*}
	Due to the boundedness of \(\a\), we have
	\begin{align*}
		I_1(f_i) &\leq \int_{\rn}\norm{\tenp{\nb f_i}-\tenp{\nb f}}\abs{\a(f_i)}\\
		&\leq M\int_{\rn}\norm{\tenp{\nb f_i}-\tenp{\nb f}}\\
		&\leq M\int_{\rn}\sum_{k=0}^{p-1}\norm{\ten{\nb f_i}{k}\odot\nb(f_i-f)\odot\ten{\nb f}{(p-1-k)}}\\
		&\leq \l_{n,p}M\int_{\rn}\sum_{k=0}^{p-1}\left(\abs{\nb f_i}^k\cdot\abs{\nb(f_i-f)}\cdot\abs{\nb f}^{p-1-k}\right).
	\end{align*}
	Now, H\"{o}lder's inequality yields
	\[I_1(f_i) \leq \l_{n,p}M\normk{p}{\nb(f_i-f)}\sum_{k=0}^{p-1}\normk{p}{\nb f_i}^k\normk{p}{\nb f}^{p-1-k}\to0,\qquad\text{as }i\to\infty.\]
	On the other hand, for every subsequence \(\xset{f_{i_j}}\subset\set{f_i}\), we are going to show that there is a subsequence \(\xset{f_{i_{j_k}}}\subset\xset{f_{i_j}}\)
	such that \(I_2(f_{i_{j_k}})\) converges to 0, to conclude that \(Z(f_i)\to Z(f)\).
	Let \(\xset{f_{i_j}}\) be a subsequence of \(\set{f_i}\), then \(\xset{f_{i_j}}\) converges to \(f\) in \(\sobop\).
	Thus, there exists a subsequence \(\xset{f_{i_{j_k}}}\subset\xset{f_{i_j}}\) with \(f_{i_{j_k}}\to f\) a.e. as \(k\to\infty\).
	Since \(\a\) is continuous, we have
	\[\a(f_{i_{j_k}})\to\a(f)\text{ a.e.},~\text{i.e. }\abs{\a(f_{i_{j_k}})-\a(f)}\to0\text{~a.e.,}\qquad\text{as}~k\to\infty.\]
	Moreover, since \(\a\) is bounded, we have
	\[\abs{\nb f}\cdot\abs{\a(f_{i_{j_k}})-\a(f)}\leq 2M\abs{\nb f}\in L^p(\rn).\]
	Hence, \(\abs{\nb f}^p\cdot\abs{\a(f_{i_{j_k}})-\a(f)}^p\) is integrable.
	By Lebesgue dominated convergence theorem, we obtain
	\begin{align*}
		\lim_{k\to\infty}\int_{\rn}\abs{\nb f}^p\cdot\abs{\a(f_{i_{j_k}})-\a(f)}^p
		= \int_{\rn}\lim_{k\to\infty}\abs{\nb f}^p\cdot\abs{\a(f_{i_{j_k}})-\a(f)}^p=0.
	\end{align*}
	Therefore,
	\[I_2(f_{i_{j_k}})\leq\l_{n,p}\int_{\rn}\abs{\nb f}^p\cdot\abs{\a(f_{i_{j_k}})-\a(f)}^p\to0,\qquad\text{as }k\to\infty,\]
	which implies the continuity.
	
	Let \(\phi\in\sln\). Then,
	\begin{align*}
		Z(f\circ\phi^{-1}) &= p^p\int_{\rn}\tenp{\nb(f\circ\phi^{-1})}\a(f\circ\phi^{-1})\\
		&= p^p\int_{\rn}\tenp{\left(\phi^{-t}\left(\nb f\circ\phi^{-1}\right)\right)}\a(f\circ\phi^{-1})\\
		&= p^p\int_{\rn}\phi^{-t}\cdot\tenp{\left(\nb f\circ\phi^{-1}\right)}\a(f\circ\phi^{-1})\\
		&= p^p\int_{\rn}\phi^{-t}\cdot\tenp{\nb f}\a(f)\\
		&= \phi^{-t}\cdot Z(f),
	\end{align*}
	which implies the \(\sln\) contravariance.
	
	Let \(\tau\) be a translation. Then,
	\begin{align*}
		Z(f\circ\tau^{-1}) &= p^p\int_{\rn}\tenp{\nb(f\circ\tau^{-1})}\a(f\circ\tau^{-1})\\
		&= p^p\int_{\rn}\tenp{\left(\nb f\circ\tau^{-1}\right)}\a(f\circ\tau^{-1})\\
		&= p^p\int_{\rn}\tenp{\nb f}\a(f)\\
		&= Z(f),
	\end{align*}
	which implies the translation invariance.
	
	It follows from \eqref{fvg} and \eqref{f^g} clearly that \(Z\) is a valuation with \(Z(0)=0\).
	
	Finally, we write \(c=(-p)^p\mop(P)\) and \(\varphi(s)=\int_0^s\a(t)(s-t)^{n-1}\df t\),
	and use Lemma \ref{thm:cal} to obtain
	\begin{align*}
		\Phi(s)=Z(s\lp)=(-p)^ps^{p-n}\mop(P)\int_0^s\a(t)(s-t)^{n-1}\df t=cs^{p-n}\varphi(s).
	\end{align*}
	For \(k=0,1,2,\dots,n-1\), by induction, we have
	\begin{equation}\label{eqn:3}
		\varphi^{(k)}(s)=\frac{(n-1)!}{(n-1-k)!}\int_0^s\a(t)(s-t)^{n-1-k}\df t,
	\end{equation}
	and \(\varphi^{(n)}(s)=(n-1)!\a(s)\). Thus,
	\begin{align}
		\Phi^{(k)}(s) &= c\sum_{j=0}^k{k \choose j}(s^{p-n})^{(k-j)}\varphi^{(j)}(s)\notag\\
		&= c\sum_{j=0}^k{k \choose j}(-1)^{k-j}\frac{(n-p+k-j-1)!}{(n-p-1)!}s^{p-n-k+j}\varphi^{(j)}(s)
		\label{eqn:4}
	\end{align}
	for \(k=0,1,2,\dots,n\). Hence, \(\Phi\in C^n(\rn)\).
	
	Since \(\a\) is bounded, there exist \(m,M\in\R\) such that \(m\leq\a(s)\leq M\) for every \(s\in\R\).
	Combined with \eqref{eqn:3}, we have
	\[\frac{(n-1)!}{(n-k)!}m\leq s^{k-n}\varphi^{(k)}(s)\leq\frac{(n-1)!}{(n-k)!}M,\]
	which implies that \(s^{k-n}\varphi^{(k)}(s)\) is bounded for \(k=0,1,2,\dots,n\).
	By \eqref{eqn:4}, we see that
	\[s^{k-p}\Phi^{(k)}(s)=c\sum_{j=0}^k{k \choose j}(-1)^{k-j}\frac{(n-p+k-j-1)!}{(n-p-1)!}s^{j-n}\varphi^{(j)}(s),\]
	and they are bounded consequently for \(k=0,1,2,\dots,n\).
\end{proof}

Based on Theorem \ref{thm:body}, the next lemma gives the first step towards the classification.
\begin{lem}\label{thm:mten}
	Let \(Z:\lip\to\symt\) be a continuous, \(\sln\) contravariant, translation invariant valuation with \(Z(0)=0\).
	Then, there is a continuous function \(\z:\R\to\R\) such that
	\[Z(s\lp)=\z(s)\mop(P)\]
	for every \(s\in\R\) and \(\lp\in\pip\).
\end{lem}
\begin{proof}
	Define \(Y_s:\polyo\to\symt\) by
	\[Y_s(P)=Z(s\lp)\]
	for every \(s\in\R\) and \(P\in\polyo\). Let \(P,Q\in\polyo\) such that \(P\cup Q\in\polyo\).
	We have \(\lp\vee\ell_Q=\ell_{P\cup Q}\) and \(\lp\wedge\ell_Q=\ell_{P\cap Q}\). Since \(Z\) is a valuation on \(\lip\), we obtain
	\begin{align*}
		Y_s(P)+Y_s(Q) &= Z(s\lp)+Z(s\ell_Q)=Z(s(\lp\vee\ell_Q))+Z(s(\lp\cap\ell_Q))\\
		&= Z(s\ell_{P\cup Q})+Z(s\ell_{P\cap Q})=Y_s(P\cup Q)+Y_s(P\cap Q),
	\end{align*}
	which shows that \(Y_s\) is a valuation. It is also clear that \(Y_s\) is Borel measurable and \(\sln\) contravariant.
	Thus, by Theorem \ref{thm:body}, there exist continuous functions \(\z_1,\z_2:\R\to\R\) such that
	\begin{equation}\label{eqn:5}
		Z(s\lp)=Y_s(P)=\z_1(s)M^{p,0}(P^\ast)+\z_2(s)\mop(P)
	\end{equation}
	for every \(s\in\R\) and \(\lp\in\pip\).
	
	Let \(P\in\polyo\). Take translations \(\tau_1,\dots,\tau_k\) such that \(\tau_i(P/k^i)\) are pairwisely disjoint for \(i=1,2,\dots,k\). Define
	\[f_k=s(\ell_{\tau_1(P/k^1)}\vee\dots\vee\ell_{\tau_k(P/k^k)}),\qquad s\in\R,\]
	where \(\ell_{\tau_i(P/k^i)}=\tau_i\circ\ell_{P/k^i}\).
	Then, \(f_k\to0\) in \(\sobop\) as \(k\to\infty\). By the valuation property, the translation invariance and \eqref{eqn:5}, we have
	\begin{align*}
		Z(f_k) &= \sum_{i=1}^kZ(s\ell_{P/k^i})\\
		&= \sum_{i=1}^k\left(\z_1(s)M^{p,0}\left((P/k^i)^\ast\right)+\z_2(s)\mop(P/k^i)\right)\\
		&= \z_1(s)\sum_{i=1}^kk^{i(n+p)}M^{p,0}(P^\ast)+\z_2(s)\sum_{i=1}^kk^{-i(n-p)}\mop(P).
	\end{align*}
	The continuity of \(Z\) and \(Z(0)=0\) force \(\z_1(s)\equiv0\) for every \(s\in\R\).
\end{proof}

Further assumptions on \(Z(s\lp)\) allow us to obtain the representation of the valuations on cone functions.
\begin{lem}\label{thm:nh}
	Let \(Z:\lip\to\symt\). If \(Z\in\bp\), and there is a continuous function \(\b:\R\to\R\) such that \(Z(s\lp)=\b(s)\mop(P)\),
	then there is a continuous and bounded function \(\a:\R\to\R\) such that
	\[Z(s\lp)=\jap((s\lp)^p).\]
\end{lem}
\begin{proof}
	First, we notice that \(\b\in C^n(\R)\), since \(\Phi(s)=Z(s\lp)=\b(s)\mop(P)\in C^n(\R)\). Next, write
	\(\b_1(s)=n(-p)^{-p}s^{n-p}\b(s)\) for \(s\in\R\). Then, \(\b_1(s)\in C^n(\R)\) as well, and
	\begin{align*}
		\b_1^{(k)}(s) &= n(-p)^{-p}\sum_{j=0}^k{k \choose j}(s^{n-p})^{(j)}\b^{(k-j)}(s)\\
		&= n(-p)^{-p}\sum_{j=0}^{\min\set{k,n-p}}{k \choose j}\frac{(n-p)!}{(n-p-j)!}s^{n-p-j}\b^{(k-j)}(s)\\
		&= n(-p)^{-p}\sum_{l=\max\set{0,k-n+p}}^k{k \choose l}\frac{(n-p)!}{(n-p-k+l)!}s^{n-p-k+l}\b^{(l)}(s)\\
		&= n(-p)^{-p}\sum_{l=\max\set{0,k-n+p}}^k{k \choose l}\frac{(n-p)!}{(n-p-k+l)!}s^{n-k}s^{l-p}\b^{(l)}(s).
	\end{align*}
	Since \(s^{l-p}\b^{(l)}(s)\) is bounded due to \(0\leq\max\set{0,k-n+p}\leq l\leq k\leq n\), we obtain \(\b_1^{(k)}(s)\to0\) as \(s\to0\)
	for \(k=0,1,2,\dots,n-1\).
	
	Now, for \(s\in\R\), let
	\[\a(s)=\frac1{n!}\b_1^{(n)}(s)=\frac{n(-p)^{-p}}{n!}\sum_{j=0}^{n-p}{n \choose j}\frac{(n-p)!}{(n-p-j)!}s^{n-p-j}\b^{(n-j)}(s).\]
	Then \(\a\) is continuous and bounded. Moreover,
	\begin{align*}
		\int_0^s \a(t)(s-t)^{n-1}\df t &= \frac1{n!}\int_0^s\b_1^{(n)}(t)(s-t)^{n-1}\df t
		= \frac1{n!}\int_0^s(s-t)^{n-1}\df\b_1^{(n-1)}(t)\\
		&= \frac1{n!}\left((s-t)^{n-1}\b_1^{(n-1)}(t)\Big|_0^s-\int_0^s\b_1^{(n-1)}(t)\df(s-t)^{n-1}\right)\\
		&= \frac1{n!}(n-1)\int_0^s\b_1^{(n-1)}(t)(s-t)^{n-2}\df t\\
		&= \cdots\\
		&= \frac1n\int_0^s\b'_1(t)\df t\\
		&= \frac1n\b_1(s),
	\end{align*}
	i.e. \(\b_1(s)=n\int_0^s\a(t)(s-t)^{n-1}\df t\). It follows that
	\[\b(s)=\frac{(-p)^p}ns^{p-n}\b_1(s)=(-p)^ps^{p-n}\int_0^s\a(t)(s-t)^{n-1}\df t.\]
	Therefore, by Lemma \ref{thm:cal}, we obtain
	\begin{align*}
		Z(s\lp)=\b(s)\mop(P)=(-p)^ps^{p-n}\mop(P)\int_0^s\a(t)(s-t)^{n-1}\df t=\jap((s\lp)^p).
	\end{align*}
\end{proof}

Now, we give the proof of Theorem \ref{thm:m2}.
\begin{proof}[Proof of Theorem \ref{thm:m2}]
	The ``if'' part follows from Lemma \ref{thm:if}. It remains to prove the ``only if '' part.
	The classification of valuations on \(\sobo\) is reduced to that on its dense subspace \(\lip\) since \(Z\) is continuous.
	Indeed, it suffices to show the classification on \(f=s\lp\) for \(s\in\R\) and \(\lp\in\pip\) due to Lemma \ref{thm:scm}.
	By Lemma \ref{thm:mten}, there is a continuous function \(\b:\R\to\R\) such that
	\[Z(s\lp)=\b(s)\mop(P)\]
	for every \(s\in\R\) and \(\lp\in\pip\). Further applying Lemma \ref{thm:nh}, we conclude that there is a continuous and bounded function 
	\(\a:\R\to\R\) such that
	\[Z(s\lp)=\jap((s\lp)^p).\]
\end{proof}

When the homogeneity is further assumed, we obtain Theorem \ref{thm:m1} as a corollary.
\begin{proof}[Proof of Theorem \ref{thm:m1}]
	Due to \eqref{eqn:jap}, the Fisher information tensor \(J^p\) is \(p\)-homogeneous. Then, we obtain the ``if'' part from Lemma \ref{thm:if}.
	
	Next, we prove the ``only if'' part. Similar to the proof of Theorem \ref{thm:m2}, it suffices to show the classification on \(f=s\lp\) for \(s\in\R\) and \(\lp\in\pip\). By Lemma \ref{thm:mten}, there is a continuous function \(\b:\R\to\R\) such that
	\[Z(s\lp)=\b(s)\mop(P).\]
	Since \(Z\) is \(p\)-homogeneous, there is \(\g\in\R\) such that \(\b(s)=\g s^p\) for \(s\in\R\).
	If we set \(\Phi(s)=Z(s\lp)\), i.e. \(\Phi(s)=\bar\g s^p\), where \(\bar\g=\g\mop(P),\)
	it is clear that \(\Phi\in C^n(\R)\) and \(s^{k-p}\Phi^{(k)}(s)\) is bounded for \(k=0,1,2,\dots,n\).
	Therefore, Lemma \ref{thm:nh} completes the proof.
\end{proof}
 
%%%%%%%%%%%%%%%%%%%%%%%%%%%%%%%%%%%%%%%%%%%%%%%%%%%%%%%%%%%%%%%%%%%%%%%%%%%%%%%%
%%%%%%%%%%%%%%%%%%%%%%%%%%%%%%%%%%%%%  Sections %%%%%%%%%%%%%%%%%%%%%%%%%%%%%%%%%%%%%
%%%%%%%%%%%%%%%%%%%%%%%%%%%%%%%%%%%%%%%%%%%%%%%%%%%%%%%%%%%%%%%%%%%%%%%%%%%%%%%%

\section*{Acknowledgment}
\addcontentsline{toc}{section}{Acknowledgment}

The authors would like to thank Lukas Parapatits for sharing some preliminary ideas on Fisher information tensors as valuations.
The work of the authors is supported in part by the National Natural Science Foundation of China (Grant No. 12471055).

%%%%%%%%%%%%%%%%%%%%%%%%%%%%%%%%%%%%%%%%%%%%%%%%%%%%%%%%%%%%%%%%%%%%%%%%%%%%%%%%
%%%%%%%%%%%%%%%%%%%%%%%%%%%%%%%%%%%%% Reference %%%%%%%%%%%%%%%%%%%%%%%%%%%%%%%%%%%%%
%%%%%%%%%%%%%%%%%%%%%%%%%%%%%%%%%%%%%%%%%%%%%%%%%%%%%%%%%%%%%%%%%%%%%%%%%%%%%%%%

\end{document}